\documentclass[12pt]{article}

\usepackage[colorlinks=true, pdfstartview=FitV, linkcolor=blue, citecolor=blue, urlcolor=blue]{hyperref}

\usepackage{amssymb,amsmath, amscd, amsthm}
\usepackage{times, verbatim}
\usepackage{graphicx, xcolor}
\usepackage{latexsym,comment,url}

\newcommand{\F}{{\mathbb F}}
\newcommand{\Fp}{{\mathbb F_p}}

\newcommand{\fbar}{\overline{f}}
\newcommand{\gbar}{\overline{g}}
\renewcommand{\hbar}{\overline{h}}

\newcommand{\PP}{{\mathbb P}}
\newcommand{\Q}{{\mathbb Q}}
\newcommand{\Qp}{{\mathbb Q}_p}
\newcommand{\R}{{\mathbb R}}

\newcommand{\rev}{{\operatorname{rev}}}
\newcommand{\spl}{{\operatorname{split}}}

\newcommand{\wt}{\operatorname{wt}}
\newcommand{\Z}{{\mathbb Z}}
\newcommand{\Zp}{{\mathbb Z}_p}

\newcommand{\stypes}{{\mathcal S}}
\newcommand{\genA}{{\mathcal A}}
\newcommand{\genB}{{\mathcal B}}
\newcommand{\genR}{{\mathcal R}}

\DeclareFontFamily{OT1}{rsfs}{}
\DeclareFontShape{OT1}{rsfs}{n}{it}{<-> rsfs10}{}
\DeclareMathAlphabet{\mathscr}{OT1}{rsfs}{n}{it}

\newtheorem{theorem}{Theorem}
\newtheorem{proposition}{Proposition}[section]
\newtheorem{corollary}[proposition]{Corollary}

\newtheorem{lemma}[proposition]{Lemma}
\newtheorem{conjecture}[proposition]{Conjecture}

\theoremstyle{definition}
\newtheorem{remark}[proposition]{Remark}

\newtheorem{definition}[proposition]{Definition}

\DeclareMathOperator{\Res}{Res}

\def\Z{{\mathbb Z}}

\def\R{{\mathbb R}}
\def\F{{\mathbb F}}

\def\Q{{\mathbb Q}}

\def\Z{{\mathbb Z}}
\def\E{{\mathbb E}}

\def\F{{\mathbb F}}
\def\Q{{\mathbb Q}}

\title{The density of polynomials of degree $n$ over $\Zp$ \\
having exactly $r$ roots in~$\Qp$}

\author{Manjul Bhargava, John Cremona, Tom Fisher,  and Stevan Gajovi\'c }

\begin{document}

\maketitle

\begin{abstract}
We determine the probability that a random  polynomial of degree $n$ over $\Z_p$ has exactly $r$ roots in $\Q_p$, and  show that it is given by a rational function of $p$ that is invariant under replacing $p$ by~$1/p$. 
\end{abstract}

\section{Introduction}
\label{sec:intro}

Let 
$f(x) = c_n x^n + c_{n-1} x^{n-1} + \cdots
+ c_0$ be a random polynomial
having  coefficients $c_0,c_1,\ldots,c_n\in\Z_p$.  
In~this paper, we determine the probability that $f$ has a root in $\Q_p$,  and more generally the probability that $f$ has 
exactly $r$ roots in $\Q_p$. 
More precisely,  we normalise the additive $p$-adic Haar measure $\mu$ on
the set of coefficients $\Zp^{n+1}$ such that $\mu(\Zp^{n+1})=1$, 
and determine the density $\mu(S_r)$ of the set $S_r$ of degree~$n$
polynomials in $\Zp[x]$ having  exactly $r$ roots in $\Qp$. We prove that this density $\mu(S_r)$ is given by a rational function $\rho^\ast(n,r;p)$ of $p$, which satisfies the remarkable  identity $$\rho^\ast(n,r;p)=\rho^\ast(n,r;1/p)$$ for all~$n$, $r$ and~$p$. We also prove that
if $X_n(p)$ is the random variable giving the number of $\Qp$-roots of a random polynomial $f\in\Zp[x]$ of degree $n$, then the $d$-th moment 
of $X_n(p)$ is independent of $n$ provided that~$n \ge 2d - 1$.

Let us now more formally define the probabilities, expectations and generating functions 
required to state our main results. 
Fix a prime $p$ and, for $0 \le r \le n$, let $\rho^*(n,r):=\rho^*(n,r;p)$ denote the
density of polynomials of degree~$n$ over $\Zp$ having  exactly $r$ roots
in $\Qp$. This is also the probability that a binary form of degree $n$
over $\Z_p$ has exactly $r$ roots in $\PP^1(\Q_p)$.
For $0 \le d \le n$, set
\begin{equation}
\label{prob-exp}
\rho(n,d) = \sum_{r=0}^n \binom{r}{d} \rho^*(n,r). 
\end{equation}
Thus $\rho(n,d)$ is the expected number of $d$-sets\footnote{We find it convenient to 
refer to a set of size $d$ as a ``$d$-set''.} of $\Qp$-roots. 
For fixed $n$, determining  $\rho(n,d)$ 
for all $d$ is equivalent to determining  $\rho^*(n,r)$ for all
$r$, via the inversion formula
\begin{equation}
\label{exp-prob}
\rho^*(n,r) = \sum_{d=0}^n (-1)^{d-r} \binom{d}{r} \rho(n,d). 
\end{equation}
Equations~\eqref{prob-exp} and~\eqref{exp-prob} are equivalent to the standard observation that a probability distribution is
determined by its moments; the formulation in terms of $d$-sets
(equivalently in terms of factorial moments)
is   most  convenient for our purposes. 

Analogous to $\rho(n,d)$, let $\alpha(n,d)$ (resp.\  $\beta(n,d)$)  denote the expected number of $d$-sets of $\Q_p$-roots of {\it monic}  polynomials of degree $n$ over $\Z_p$ (resp.\ monic polynomials of degree $n$ over $\Z_p$ that reduce to $x^n$ modulo $p$).   
Define the  generating functions:
\begin{align*}
\genA_d(t) &= (1-t) \sum_{n=0}^\infty \alpha(n,d) t^n; \\
\genB_d(t) &= (1-t) \sum_{n=0}^\infty \beta(n,d) t^n; 
\\
 \genR_d(t) &= (1-t)(1-pt) \sum_{n=0}^\infty
  (p^n+p^{n-1}+\cdots+1)
 \rho(n,d) 
 t^n. 
 \end{align*}
\noindent 
Then we prove the following theorem. 

\begin{theorem}
\label{thm:totsplit} Let $p$ be a prime number and $n$, $d$ any integers such that $0\le d\le n$.  Then: 
\begin{itemize}\item[{\rm (a)}]
For fixed $n$ and $d$, the expectations $\alpha(n,d;p)$,
$\beta(n,d;p)$ and $\rho(n,d;p)$ are rational functions of~$p$, which satisfy the identities:
\begin{align}
\label{fe:rho}
\rho(n,d;p) &= \rho(n,d;1/p); \\[.05in]
\label{fe:alphabeta}
\alpha(n,d;p) &= \beta(n,d;1/p). 
\end{align}
\item[{\rm (b)}]
We have the following power series identities in 
two variables $t$ and $u$: \begin{align}
\label{alphatobeta:mom}
\displaystyle \sum_{d=0}^\infty \genA_d(pt) u^d &= 
   \displaystyle\left( \sum_{d=0}^\infty \genB_d(t) u^d \right)^p; \\[.075in]
\label{rho:mom}
\displaystyle\sum_{d=0}^\infty \genR_d(t) u^d &=   
   \displaystyle\left( \sum_{d=0}^\infty \genA_d(pt) u^d \right)
   \left( \sum_{d=0}^\infty \genB_d(t) u^d \right)= 
   \left( \sum_{d=0}^\infty \genB_d(t) u^d \right)^{p+1}; 
\\[.165in]
\label{betatoalpha:mom}
\genB_d(t) - t \genB_d(t/p) &= \Phi \left( \genA_d(t) - t \genA_d(pt) \right),
\end{align}
where $\Phi$ is the operator on power series that
multiplies the coefficient of $t^n$ by $p^{-\binom{n}{2}}$.
\item[{\rm (c)}]
The power series $\genA_d$, $\genB_d$ and
$\genR_d$ are in fact  polynomials of degree at most $2d$. Moreover,  we have  $\alpha(n,d) = \genA_d(1)$ and  $\beta(n,d) = \genB_d(1)$ for
$n \ge 2d$, and $\rho(n,d) = \genR_d(1)$ for $n \ge 2d - 1$. 
\linebreak
Thus the expectations $\alpha(n,d)$, $\beta(n,d)$, and $\rho(n,d)$ are independent of $n$ provided that $n$ is sufficiently large relative to $d$.
\end{itemize}
\end{theorem}
\noindent
We observe that  $\genA_d$ and $\genB_d$ (for $d = 0,1,2, \ldots$)
are the unique power series satisfying
the relations 
\eqref{alphatobeta:mom} and \eqref{betatoalpha:mom}
together with the 
requirements that $\genA_d$ and $\genB_d$ are $O(t^d)$, $\genA_0 = \genB_0 = 1$
and $\genA_1$ and $\genB_1$ are $t + O(t^2)$. This last requirement is needed,
since otherwise we could replace $\genA_d$ and $\genB_d$
by $\lambda^d \genA_d$ and $\lambda^d \genB_d$ where $\lambda$ is a constant.
This uniqueness statement is easily proved by induction on $d$ and $n$.
The power series $\genR_d$ are then uniquely determined by~\eqref{rho:mom}.

While  we have stated all our results above 
in terms of the ring $\Zp$, the generalisation to any complete discrete valuation ring with 
finite residue field (as considered in \cite{Buhler-et-al}) is immediate.

\subsection{Relation to previous work}
\label{subsec:previous}
The study of the distribution of the number of zeros of random polynomials has a long and interesting history.  Over the real numbers, the study goes back to at least Bloch and P\'olya~\cite{BP}, who proved asymptotic bounds on the expected number of real zeros of  polynomials of degree~$n$ that have coefficients independently and uniformly distributed in $\{-1,0,1\}$.  Further significant advances on the problem were made by Littlewood and Offord~\cite{LO,LO2,LO3} for various other distributions on the coefficients. 

An exact formula for the expected number of real zeros of a random
degree $n$ polynomial over~$\R$---whose coefficients are each
identically, independently, and normally distributed with mean
zero---was first determined in the landmark 1943 work of
Kac~\cite{Kac}, which  influenced much of the extensive work to
follow. In particular, in~1974, Maslova~\cite{Maslova,Maslova2}
determined asymptotically all higher moments for the number of zeros
of a random real Kac polynomial in the limit as the degree $n$ tends
to infinity.  For excellent surveys of the literature and further
related results and references regarding the number of real zeros of
random real polynomials, see the works of Dembo, Poonen, Shao, and
Zeitouni~\cite[\S1.1]{DPSZ}  and of Nguyen and Vu~\cite[\S1]{NV}.

The corresponding problems and methods over $p$-adic fields were first
considered by Evans~\cite{Evans}, who determined, for suitably random
families of $d$ polynomials in $d$ variables over $\Z_p$, the expected
number of common zeros in $\Z_p^d$.
In the case $d=1$, these results were taken further by Buhler,
Goldstein, Moews, and Rosenberg~\cite{Buhler-et-al},
Caruso~\cite{Caruso}, Limmer~\cite{Limmer},
Shmueli~\cite{Shmueli2021}, and Weiss~\cite{Weiss2013}. These papers
were concerned primarily with determining the expected number of roots
for polynomials of degree $n$ over the $p$-adics, the $n$-th factorial
moments for polynomials of degree $n$, or all moments for polynomials
of degree $n\leq 3$.

The current paper gives a method for computing all moments for the
number of zeros of random $p$-adic polynomials of degree $n$ in one
variable for any degree~$n$.  Indeed, Theorem~\ref{thm:totsplit},
together with the uniqueness statement that follows it, enables us to
explicitly compute the probabilities and moments $\rho^*(n,r)$,
$\rho(n,d)$, $\alpha(n,d)$, and $\beta(n,d)$ for any values of $n$,
$r$, and $d$.  We may similarly compute the analogues $\alpha^*(n,r)$
and $\beta^*(n,r)$ of $\rho^*(n,r)$; i.e., $\alpha^*(n,r)$
(resp.\ $\beta^*(n,r)$) denotes the probability that a random {\it
  monic} polynomial of degree $n$ (resp.\ monic polynomial reducing to
$x^n$ modulo $p$) has exactly~$r$ roots over~$\Q_p$
(equivalently,~$\Z_p$). Indeed, the formulas (\ref{prob-exp}) and
(\ref{exp-prob}) continue to hold when the symbol~$\rho$ is replaced
by~$\alpha$~(resp.~$\beta$).  In~particular, we deduce
from~(\ref{exp-prob}) that $\rho^*(n,r)$, $\alpha^*(n,r)$, and
$\beta^*(n,r)$ all~satisfy the same symmetry properties (\ref{fe:rho})
and (\ref{fe:alphabeta}) as their unstarred counterparts.

We thus recover all  previously known values of~$\rho^*$, $\alpha^*$,
$\beta^*$, $\rho$, $\alpha$, and  $\beta$, including  that
$\rho(n,1)\!=\!1$ for all~$n$ (a result independently due to
Caruso~\cite{Caruso}, and Kulkarni and Lerario~\cite{KL});
that $\alpha(n,1)=p/(p+1)$ (a result of Shmueli~\cite{Shmueli2021});
and the values of~$\rho^*(n,n)$ for all~$n$ (as determined by Buhler,
Goldstein, Moews, and Rosenberg~\cite{Buhler-et-al}).

There remain three striking aspects of our formulas in
Theorem~\ref{thm:totsplit} that call for explanation: 1) they are all
rational functions in $p$ that are independent of $p$ and are valid for
all primes $p$ (including for small primes and primes $p\mid n$); 2) they
satisfy a symmetry $p\leftrightarrow 1/p$; and 3) they stabilize for large
$n$.

Properties 1) and 2) also occurred in earlier work of the first three
authors (see for example~\cite{BCFJK}). Property 1) may be related to the
work of Denef and Loeser~\cite{DL} (see also Pas~\cite{P}), at least for
sufficiently large $p$. Regarding Property 2), the expectations we study
may be expressed as $p$-adic integrals (see, e.g.,
Section~\ref{sec:pfthm3}), raising the interesting possibility that there
might be a common explanation for the $p \leftrightarrow 1/p$ symmetries
occurring in Theorem~\ref{thm:totsplit}~\!(a) and the functional equations
for certain zeta functions established by Denef and Meuser~\cite{DM}, who
count the number of zeros of a homogeneous polynomial mod $p^m$, and by du
Sautoy and Lubotzky~\cite{dSL}, who count finite index subgroups of a
nilpotent group. Finally, regarding Property 3), there is the interesting
possibility that the independence of $n$ established in
Theorem~\ref{thm:totsplit}~\!(c) might fit into the framework of
representation stability as initiated in the work of Church, Ellenberg and
Farb~\cite{CFB}. We believe it is an exciting problem to understand these
phenomena and their potential relations with the aforementioned works.

\subsection{Examples}

We illustrate some particularly interesting cases of Theorem~\ref{thm:totsplit}.

\subsubsection{The expected number of roots of a random $p$-adic polynomial}

By definition, the quantities $\rho(n,1)$, $\alpha(n,1)$, and $\beta(n,1)$ represent the expected number of roots over~$\Q_p$ of a random  polynomial over $\Zp$ of degree~$n$, a random monic polynomial over~$\Zp$ of degree~$n$, and~a~random monic polynomial over $\Zp$ of degree $n$ reducing to $x^n$ (mod~$p$), respectively. 

Setting $d = 1$, we compute
\[\genA_1(t) = t - \frac{1}{p+1} t^2, \quad
\genB_1(t) = t - \frac{p}{p + 1} t^2, \quad  
\genR_1(t) = (p + 1) t - p t^2.\] Therefore, 
\[ \alpha(n,1) = \left\{ \begin{array}{cl} 
1 & \text{ if } n = 1, \\[.025in]
\displaystyle \frac{p}{p+1} & \text{ if } n \ge 2, 
\end{array} \right. \quad
\beta(n,1) = \left\{ \begin{array}{cl} 
1 & \text{ if } n = 1, \\[.025in]
\displaystyle \frac1{p+1} & \text{ if } n \ge 2, 
\end{array} \right.
\]
and
$$\rho(n,1) = 1 \;\; \mbox{for all $n \ge 1$.}$$ This recovers, in
particular, the aforementioned results of Caruso~\cite{Caruso} and
Kulkarni and Lerario~\cite{KL} on the values of $\rho(n,1)$, and of
Shmueli~\cite{Shmueli2021} on $\alpha(n,1)$, who obtained
them via quite different methods (though their methods are related to
those used by Kac~\cite{Kac} cited above).

\subsubsection{The second moment of the number of $\Q_p$-roots of a random $p$-adic polynomial}

Next, we determine the expected number of 2-sets (i.e., unordered pairs) of $\Q_p$-roots of  a polynomial over $\Zp$ of degree $n$.
Setting $d=2$, we compute
\begin{align*}
 2 \genA_2(t) &= (p/(p + 1)) t^2
       - p (p + 1) (2 p^3 + p + 1) \eta t^3
       + p^4 \eta t^4, \\[.03in]
 2 \genB_2(t) &= (1/(p + 1)) t^2
      - p (p + 1) (p^3 + p^2 + 2) \eta t^3
      + p^2 \eta t^4, \\[.03in]
 2 \genR_2(t) &= (p^2 + p + 1) t^2
      - p (p + 1)^3 (2 p^4 + 3 p^2 + 2) \eta t^3
      + p^2 (p + 1)^2 (p^4 + p^2 + 1) \eta t^4,
\end{align*}
where $\eta = 1/((p + 1)^2(p^4+p^3+p^2+p + 1))$. Therefore, 
\[ 2\alpha(n,2) = \left\{ \begin{array}{ll} 
p/(p+1) & \text{ if } n = 2, \\ p^3 (p^3 + 1) \eta & \text{ if } n = 3, \\
p^3 (p^3 + p + 1) \eta & \text{ if } n \ge 4, 
 \end{array} \right. \qquad
2\beta(n,2) = \left\{ \begin{array}{ll} 
1/(p+1) & \text{ if } n = 2, \\
(p^3 + 1) \eta & \text{ if } n = 3, \\
(p^3 + p^2 + 1) \eta & \text{ if } n \ge 4,
 \end{array} \right. 
\]
and 
$$ \rho(2,2)=1/2,\quad
2\rho(n,2) = (p^2 + 1)^2/(p^4+p^3+p^2+p+1) \;\;\mbox{for all $n \ge 3$.}$$ 

\noindent There is no difficulty in extending these calculations to larger values of $d$.

\subsubsection{The density of $p$-adic polynomials of degree $n$ having $r$ roots} 
\label{sec:listanswers}

Once we have computed the expectations $\rho(n,d)$, $\alpha(n,d)$, and $\beta(n,d)$, we may 
use~\eqref{exp-prob} and its analogues for $\alpha$ and $\beta$ to compute the probabilities $\rho^*(n,r)$,
$\alpha^*(n,r)$, and
$\beta^*(n,r)$.
Since the probability of a repeated root is zero, we always
have 
$\rho^*(n,n-1) = \alpha^*(n,n-1)=
\beta^*(n,n-1)=0$.

For $n=2$ and $3$, the probabilities $\rho^*(n,r)$  can already be deduced from 
results in~\cite{BCFJK}, \cite{Buhler-et-al} and~\cite{Caruso}. 
Namely, we have $$\rho^*(2,0) = \rho^*(2,2) = 1/2,$$  
and
$$\rho^*(3,0) = 2 \gamma,\;\, \rho^*(3,1) = 1 - 3 \gamma,\;\, \rho^*(3,3) = \gamma,$$
where 
\[ \gamma = \frac{(p^2 + 1)^2}{6(p^4+p^3+p^2+p+ 1)}. \]

For quartic  polynomials in  $\Zp[x]$, the probability of  having
$0,1,2$ or $4$ roots in $\Qp$ is given by
\begin{align*}
\rho^*(4,0) &= \frac{\delta}{8}(3 p^{12} +
    5 p^{11} + 8 p^{10} + 12 p^9 + 13 p^8 + 12 p^7 + 17 p^6 + 12 p^5 + 13 p^4  + 12 p^3 + 8 p^2 + 5 p + 3),\\[.03in]
\rho^*(4,1) &= \frac{\delta}{3} (p^{12} + 2 p^{11}
   + 4 p^{10} + 3 p^9 + 6 p^8 + 7 p^7 + 2 p^6 + 7 p^5 + 6 p^4 + 3 p^3  + 4 p^2 + 2 p + 1),\\[.03in]
\rho^*(4,2) &= \frac{\delta}{4} (p^{12} + 3 p^{11}
    + 2 p^{10} + 6 p^9 + 5 p^8 + 4 p^7 + 9 p^6 + 4 p^5 + 5 p^4 + 6 p^3  + 2 p^2 + 3 p + 1),    \\[.03in]
\rho^*(4,4) &= \frac{\delta}{24} ( p^{12} - p^{11} +
    4 p^{10} + 3 p^8 + 4 p^7 - p^6 + 4 p^5 + 3 p^4 + 4 p^2 - p + 1),  
\end{align*}
where 
\[ \delta = \frac{(p-1)^2}{(p^5-1)(p^9-1)}.\]
\noindent
The last of these probabilities, $\rho^*(4,4)$, was determined in \cite{Buhler-et-al}, where
it is denoted $r_4^{\operatorname{nm}}$. As predicted by Theorem~\ref{thm:totsplit}(a),
the sequence of coefficients in each numerator and in each denominator is palindromic.
Again, there is no difficulty in computing $\rho^*(n,r)$ for larger 
values of $n$. 

For $n=2$ and $3$, the probabilities $\alpha^*(n,r)$ 
were computed by Limmer~\cite[p.~27]{Limmer} and Weiss~\cite[Theorem 5.3]{Weiss2013},
who only considered primes $p > n$. Our work shows that the same formulas hold for all primes~$p$.
Namely, we have
\begin{align*}
\alpha^*(2,0) &= \frac12\,\frac{p + 2}{\hspace{.01in}p + 1},
\!\!\!\!\qquad
\alpha^*(2,2) = \frac12\, \frac{p}{p + 1};\\[.055in]
 \alpha^*(3,0) &= \frac{1}{3}\,\frac{p^4 + p^3 + 3p^2 + 3}{p^4+p^3+p^2+p+1},\\[.0125in]
    \alpha^*(3,1) &= \frac{1}{2}\,\frac{p^5 + 3p^4 + p^3 + 2p^2 + 2p}{(p+1)(p^4+p^3+p^2+p+1)},\\[.0125in] 
    \alpha^*(3,3) &= \frac{1}{6}\,\frac{p^5 - p^4 + p^3}{(p+1)(p^4+p^3+p^2+p+1)}.
\end{align*}

For monic quartic polynomials in $\Zp[x]$, the probability of  having
$0$, $1$, $2$ or $4$ roots in $\Zp$ is given by 
\begin{align*}
    \alpha^*(4,0) &= \:\frac{1}{8}\;\frac{3p^{11} + 8p^{10} + 6p^9 + 2p^8 - 3p^6 + 4p^5 - 4p^3 - 8p - 8}{(p+1)^2(p^9-1)},\\[.03in]
    \alpha^*(4,1) &= \:\frac{1}{3}\;\frac{p^{14} + 2p^{12} - 6p^{11} + 9p^{10} - 9p^9 + 2p^8 + 3p^7 - 2p^6 - 3p^5 + 3p^4 - 3p^2 + 3p}{(p^5-1)(p^9-1)},\\[.03in]
    \alpha^*(4,2) &= \:\frac{1}{4}\;\frac{p^{16} + 2p^{15} - 4p^{14} + 2p^{13} + 2p^{12} - 6p^{11} + 4p^{10} + 2p^9 - 6p^8 + 2p^7 + p^6 - 2p^5 + 2p^3}{(p+1)^2(p^5-1)(p^9-1)},\\[.03\in]
    \alpha^*(4,4) &= \frac{1}{24}\,\frac{p^{16} - 4p^{15} + 6p^{14} - 2p^{13} - 4p^{12} + 6p^{11} - 4p^{10} - 2p^9 + 6p^8 - 4p^7 + p^6}{(p+1)^2(p^5-1)(p^9-1)}.
\end{align*}
By the analogue of (\ref{fe:alphabeta}) for $\alpha^*$ and $\beta^*$, we may obtain the values of $\beta^*$ from those of $\alpha^*$ by substituting~$1/p$ for~$p$.
\label{starvalues}

\subsubsection{The density of $p$-adic polynomials that split completely}
\label{sec:buhler}

The quantities $\rho(n,n)$ and $\alpha(n,n)$ represent the probabilities that a
(general or monic) polynomial of degree $n$ over $\Zp$ splits completely
over $\Qp$. These probabilities were previously computed 
by Buhler, Goldstein, Moews, and Rosenberg~\cite{Buhler-et-al}. We may recover these probabilities from Theorem~\ref{thm:totsplit} as follows.  If we replace 
$\genA_d$, $\genB_d$, and $\genR_d$ by their coefficients of $t^d$ (these being
the terms of lowest degree in $t$), then Theorem~\ref{thm:totsplit}(b) 
reduces to 
\begin{align}
\label{alphatobeta:s}
\sum_{n=0}^\infty \alpha(n,n) (pt)^n &= 
   \left( \sum_{n=0}^\infty \beta(n,n) t^n \right)^p \\
\label{rho:s}
\!\!\!\!\!\!\!
\!\!\!\!\!\!\!
\!\!\!\!\!\!\!
\sum_{n=0}^\infty (p^n+p^{n-1} + \cdots + 1) \rho(n,n) t^n &=  
   \left( \sum_{n=0}^\infty \beta(n,n) t^n \right)^{p+1} \\[.1in]
\label{betatoalpha:s}
\beta(n,n) &= p^{-\binom{n}{2}} \alpha(n,n),
\end{align}
from which one can inductively compute $\rho(n,n)$,  $\alpha(n,n)$, and $\beta(n,n)$ for all $n$. In~\cite{Buhler-et-al},
Buhler \textit{et al.}~write $r_n^{\operatorname{nm}}$, $r_n$, and $p^n s_n$ for $\rho(n,n)$, $\alpha(n,n)$,  and $\beta(n,n)$, respectively. Our equations~\eqref{alphatobeta:s}
and~\eqref{rho:s} appear as Equations (1-2) and (3-1) in their paper; and their Lemma~4.1(iv), which states that \linebreak $r_n(q) = r_n(1/q) q^{\binom{n}{2}}$,
follows by combining our general  Equation~\eqref{fe:alphabeta}  with~\eqref{betatoalpha:s}.
The explicit values of $\rho(n,n) = \rho^*(n,n)$, $\:\alpha(n,n)=\alpha^*(n,n)\,$, and $\beta(n,n)=\beta^*(n,n)$ for $n \le 4$ were recorded  in~\S\ref{starvalues}. 

\subsubsection{The density of $p$-adic polynomials with a root}\label{sec:originalproblem}
We may also compute $1 - \rho^*(n,0)$, the probability that a polynomial of degree $n$ over $\Zp$ has at least one root over $\Q_p$.
Indeed, as a special case of~\eqref{exp-prob}, we have
$\rho^*(n,0) = \sum_{d=0}^n (-1)^d \rho(n,d)$, and likewise for
the $\alpha$'s and $\beta$'s. In terms of generating functions, we have
\[ \genA^*(t) := (1-t) \sum_{n=0}^\infty \alpha^*(n,0) t^n 
= \sum_{d=0}^\infty (-1)^d \genA_d(t) \]
\[ \genB^*(t) := (1-t) \sum_{n=0}^\infty \beta^*(n,0) t^n 
= \sum_{d=0}^\infty (-1)^d \genB_d(t) \]
and
\[ \genR^*(t) := (1-t)(1-pt) \sum_{n=0}^\infty (p^n+p^{n-1} + \cdots + 1) 
\rho^*(n,0) t^n = \sum_{d=0}^\infty (-1)^d \genR_d(t). \]
Specialising Theorem~\ref{thm:totsplit}(b) by setting $u=-1$ gives
\begin{align}
\label{alphatobeta:noroot}
\genA^*(pt) &= \genB^*(t)^p \\
\label{rho:noroot}
\genR^*(t) &= \genA^*(pt) \genB^*(t) = \genB^*(t)^{p+1}\\
\label{betatoalpha:noroot}
\genB^*(t) - t \genB^*(t/p) &= \Phi( \genA^*(t) - t \genB^*(pt) )
\end{align}
where $\Phi$ is as before.

We may therefore  use~\eqref{alphatobeta:noroot} and~\eqref{betatoalpha:noroot}
to recursively solve for $\alpha^*(n,0)$ and $\beta^*(n,0)$, and then
compute $\rho^*(n,0)$ using~\eqref{rho:noroot}. The explicit values of $\alpha^*(n,0)$, $\beta^*(n,0)$, and $\rho^*(n,0)$ for $n\leq 4$ were recorded in~\S\ref{starvalues}. 

\subsubsection{Large $p$ limits}\label{largep}

We note that $\alpha(n,d)$, $\rho(n,d)$, $\alpha^*(n,r)$, and $\rho^*(n,r)$ are rational functions in $p$ whose numerators and denominators have the same degree. 
Hence, for fixed $n$, $d$,  and $r$, 
we may compute the limits of these functions as $p$ tends to infinity. Meanwhile,  $\beta(n,d)$ and $\beta^*(n,r) $ 
are rational functions in $p$ whose denominator has higher degree than the numerator in most cases. Thus, a correction factor of a power of $p$ is needed to make the limit finite and nonzero. We have the following proposition.

\begin{proposition} \label{prop:largep}
\phantom{hello}

\begin{itemize}\item[{\rm (a)}]  Let $0 \leq d \leq n$ 
be integers, and let $k = \min(d+1,n)$. Then 
\begin{equation*} 
\lim_{p\rightarrow \infty}\alpha(n,d)=\lim_{p\rightarrow \infty}\rho(n,d)=\lim_{p\rightarrow \infty}p^{\binom{k}{2}}\beta(n,d)=\dfrac{1}{d!}. \end{equation*}
\item[{\rm (b)}]  
Let $0 \leq r \leq n$ 
be integers.  Then
\[\lim_{p\rightarrow \infty}\rho^*(n,r) = \lim_{p\to\infty}\alpha^*(n,r)=\sum_{d=0}^n (-1)^{d-r} \binom{d}{r} \dfrac{1}{d!}=\frac{1}{r!} \sum_{d=0}^{n-r} (-1)^{d}\dfrac{1}{d!}.
\]
Hence, if we also let $n\to\infty$, we obtain  $$\lim_{n\rightarrow \infty}\lim_{p\rightarrow \infty}\rho^*(n,r)=\lim_{n\rightarrow \infty}\lim_{p\rightarrow \infty}\alpha^*(n,r)=\frac{1}{r!}e^{-1}.$$ 

\item[{\rm (c)}]
Finally, let $0 \leq r \leq n$ be integers, and let $k = \min(r+1,n)$.
If $r \neq n-1$ then
\begin{equation*} 
\lim_{p\rightarrow \infty}p^{\binom{k}{2}}\beta^*(n,r)=\frac{1}{r!}.
\end{equation*}
\end{itemize}
\end{proposition}

\vspace{.075in}
We prove these claims in Section~\ref{section4}.

\subsection{A general conjecture}

Theorem~\ref{thm:totsplit}(a)  naturally leads us to formulate a much more general conjecture. Namely, we conjecture that the density of polynomials of degree $n$ over $\Z_p$ cutting out \'etale extensions of $\Q_p$ of degree $n$ in which $p$ has {\it any} given splitting type is a rational function of $p$ satisfying the identities (\ref{fe:rho}) and~(\ref{fe:alphabeta}). 

Recall that a {\em splitting type of degree $n$} is a tuple $\sigma =
(d_1^{e_1}\,d_2^{e_2}\,\cdots\,d_t^{e_t})$, where the $d_j$ and $e_j$
are positive integers satisfying $\sum d_j e_j=n$.  We allow repeats
in the list of symbols~$d_j^{e_j}$, but the order in which they appear
does not matter. To make it clear when two splitting types are
the same, we could for example order the pairs $(d_j,e_j)$
lexicographically. Exponents~$e_j=1$ may be omitted.  

For an \'etale extension $K/\Q_p$ of degree~$n$, we define the symbol $(K,p)$ to be the splitting type 
$\sigma =
(d_1^{e_1}\,d_2^{e_2}\,\cdots\,d_t^{e_t})$ if $p$ factors in $K$ as 
$P_1^{e_1}P_2^{e_2}\cdots P_t^{e_t}$, where $P_1,P_2,\ldots,P_t$ are primes in $K$ having residue field degrees $d_1,d_2,\ldots,d_t$, respectively. We say that {\it $p$ has splitting type $\sigma$ in $K$} if $(K,p)=\sigma$.

We then make the following conjecture.

\begin{conjecture}
\label{genconj}
Let $\sigma$ be any splitting type of degree $n$, and  set 
\begin{align*}
\rho(n,\sigma;p) &:= \mbox{\rm density of polynomials  $f\in\Z_p[x]$ of degree $n$}\\
&\phantom{=}\;\;\mbox{ \rm such that $K\!:=\!\Q_p[x]/f(x)$ is \'etale over $\Q_p$ 
and 
$(K,p)=\sigma$, 
}\\[.1in]
\alpha(n,\sigma;p) &:=\mbox{\rm density of monic polynomials $f\in\Zp[x]$ of degree $n$}\\
&\phantom{=}\;\;\mbox{ \rm such that $K\!:=\!\Q_p[x]/f(x)$ is \'etale over $\Q_p$ 
and 
$(K,p)=\sigma$,
}\\[.1in]
\beta(n,\sigma;p) &:= \mbox{\rm density of monic polynomials $f\in\Zp[x]$ of degree $n$ with 
$f(x)\equiv x^n$ (mod $p$)}\\
&\phantom{=}\;\;\mbox{ \rm such that $K\!:=\!\Q_p[x]/f(x)$ is \'etale over $\Q_p$ 
and 
$(K,p)=\sigma$.
}
\end{align*}
Then $\rho(n,\sigma;p)$, $\alpha(n,\sigma;p)$, and $\beta(n,\sigma;p)$ are rational functions of $p$ and satisfy the identities:
\begin{align}
\label{fe:rho2}
  \rho(n,\sigma;p) &= \rho(n,\sigma;1/p);  \\
  \label{fe:alphabeta2}
 \alpha(n,\sigma;p)&=\beta(n,\sigma;1/p).
\end{align}
\end{conjecture}

\vspace{.025in}
\noindent
We have proven 
that Conjecture~\ref{genconj} holds in the quadratic and cubic cases. For example, 
\begin{align*}
\rho(2,(11);p)    &= 1/2\\[.02in]
\rho(2,(2);p)     &= 1/2 - p/(p^2 + p + 1)\\[.02in]
\rho(2,(1^2);p)   &= p/(p^2 + p + 1)\\[.1in]
\rho(3,(111);p)   &= (1/6) (p^4 + 2p^2 + 1)/(p^4 + p^3 + p^2 + p + 1)\\[.02in]
\rho(3,(12);p)    &= (1/2) (p^4+1)/(p^4 + p^3 + p^2 + p + 1)\\[.02in]
\rho(3,(3);p)     &= (1/3) (p^4 - p^2 + 1)/(p^4 + p^3 + p^2 + p + 1)\\[.02in]
\rho(3,(1^2 1);p) &= (p^3 + p)/(p^4 + p^3 + p^2 + p + 1)\\[.02in]
\rho(3,(1^3);p)   &= p^2/(p^4 + p^3 + p^2 + p + 1).
\end{align*}
Note again that the numerators and denominators are all palindromic, and thus these expressions   satisfy~(\ref{fe:rho2}). Analogous formulas hold for the $\alpha$'s and $\beta$'s that satisfy (\ref{fe:alphabeta2}). 
In particular, these formulas hold for all $p$, including $p=2$ and $p=3$.

Theorem~\ref{thm:totsplit}(a) may also be viewed as a special case of Conjecture~\ref{genconj}, since the density $\rho^*(n,r;p)$ of polynomials of degree $n$ over $\Z_p$ having exactly $r$ roots over $\Q_p$ is simply the sum of the densities $\rho(n,\sigma;p)$ over all splitting types $\sigma$ having exactly $r$ 1's (and similarly for the $\alpha$'s and $\beta$'s); thus if the equalities (\ref{fe:rho2}) and (\ref{fe:alphabeta2})  hold for all $\rho(n,\sigma;p)$, then they will also hold for  
$\rho^*(n,r)$ and $\rho(n,d)$ (and similarly for the $\alpha$'s and $\beta$'s), implying Theorem~\ref{thm:totsplit}(a). 

\subsection{Methods and organization of the paper}\label{methods}

In Section~\ref{sec:prelim}, we explain some preliminaries needed for the proof of Theorem~\ref{thm:totsplit}, regarding counts of polynomials in $\F_p[x]$ having given factorization types, power series identities involving these counts, resultants of polynomials over $\Z_p$, and  explicit forms of Hensel's lemma for polynomial factorization. 

In Section~\ref{sec:proofof12}, we then turn to the proof of Theorem~\ref{thm:totsplit}. We first explain how Theorem~\ref{thm:totsplit}(b) easily implies Theorem~\ref{thm:totsplit}(a). To prove Theorem~\ref{thm:totsplit}(b), we begin by writing the $\alpha(n,d)$ in terms of the 
$\beta(n',d')$ for $n' \le n$ and $d' \le d$. This involves considering how a 
monic polynomial over $\Zp$ factors mod $p$ and showing that the random variables given by the number of $\Zp$-roots above each $\Fp$-root are independent.
The answers may be expressed in terms of the generating functions $\genA_d$ and $\genB_d$ as 
\begin{equation}
\label{AfromB}
\begin{aligned}
           \genA_1(pt) &= p \genB_1(t) \\
           \genA_2(pt) &= p \genB_2(t) + \frac{1}{2} p(p-1) \genB_1(t)^2 \\
           \genA_3(pt) &= p \genB_3(t) + p (p-1) \genB_1(t) \genB_2(t) + \frac{1}{6} p(p-1)(p-2) \genB_1(t)^3 \\
             \quad &\;\;\vdots \hspace{10em}  
\end{aligned}
\end{equation}
which may be expressed more succinctly in the form~\eqref{alphatobeta:mom}.
We then explain how to write the $\beta(n,d)$ in terms of the
$\alpha(n',d)$ for $n' \le n$. This is proved by making substitutions of the form
$x \leftarrow px$, and analysing the  valuations of the resulting  coefficients; the relation we obtain is expressed succinctly in the form~\eqref{betatoalpha:mom}.
These two types of relations allow us then to recursively solve for the $\alpha$'s and $\beta$'s.
We then write the $\rho$'s in terms of the $\alpha$'s and $\beta$'s,  using another related independence result, and 
the relations we thereby obtain are expressed succinctly in the form~\eqref{rho:mom}, completing the proof of Theorem~\ref{thm:totsplit}(b). 

As previously noted, Theorem~\ref{thm:totsplit}(b) gives
a way to compute  the power series 
$\genA_d$, $\genB_d$ and~$\genR_d$ for each $d$. However, it does not seem to give any way of showing that 
these are in fact  polynomials for all~$d$. In establishing  Theorem~\ref{thm:totsplit}(c), we thus use a different technique to prove 
the stabilisation result for the~$\alpha$'s, or equivalently, that $\genA_d$ is a polynomial of
degree at most $2d$. We could also give a similar proof of the corresponding result for the
$\beta$'s, but there is no need, since it follows from that for the $\alpha$'s, 
using either \eqref{fe:alphabeta} or \eqref{AfromB}.
 
Once we have shown  that $\genA_d$ and $\genB_d$ are polynomials of degree at most $2d$, the
same result for $\genR_d$ then follows by~\eqref{rho:mom}. 
This is not sufficient to prove the stabilisation
result for the $\rho$'s, since the definition of $\genR_d$ involves additional factors.
However, a variant of the ideas used to show that $\genA_d$ is a polynomial also show that 
$\genA_d(1) = \genA_d(p)$, and from this we deduce the stabilisation result for the~$\rho$'s.

Finally, in Section~\ref{section4}, we prove the asymptotic results contained in \S\ref{largep}.

\section{Preliminaries}
\label{sec:prelim}

\subsection{Basic notation}
For a ring $R$, let $R[x]$ denote the ring of univariate polynomials
over~$R$, and for $n\ge0$, let $R[x]_n$ denote the subset of
polynomials of  degree at most~$n$, and  
$R[x]_n^1$ the subset of monic polynomials of degree~$n$.

In the case  $R=\Z_p$, we identify $\Zp[x]_n^1$ with $\Zp^n$ via $$x^n+\sum_{i=0}^{n-1}a_ix^i
\leftrightarrow (a_0,a_1,\dots,a_{n-1}),$$ and thereby use the usual $p$-adic measure on subsets of~$\Zp[x]_n^1$ inherited via this identification.

For $f\in\Zp[x]$, we denote by $\overline{f}$ its image under
reduction modulo~$p$ in $\Fp[x]$.  A polynomial with coefficients in~$\Zp$ is \emph{primitive} if not all
its coefficients are divisible by~$p$, that is, if $\fbar\not=0$.  For a primitive polynomial~$f\in\Z_p[x]$, we define the \emph{reduced degree} of~$f$ to be
$\deg(\fbar)$.  Hence $\deg(\fbar)\le\deg(f)$, with equality if
and only if the leading coefficient of~$f$ is a unit.

\subsection{Counts involving splitting types of polynomials over \texorpdfstring{$\F_p$}{Fp}}\label{sec:patterns}

We will require  expressions for the number of monic polynomials 
in $\F_p[x]$ that factor as a product of irreducible polynomials with given 
degrees and multiplicities. These counts, and the corresponding 
probabilities for a random polynomial to have given
factorization types, are  collected in this subsection.

To this end, let 
$\stypes(n)$ denote the set of all splitting types of degree $n$.
Thus, for example, $\stypes(2) = \{(1\,1), (1^2), (2)\}$ has three elements,
$\stypes(3)$ has five elements, and $\stypes(4)$ has~$11$.

We say that a monic  polynomial $f$ in $\F_p[x]$ of degree $n$ has {\it splitting type}  
$(d_1^{e_1}\,d_2^{e_2}\,\cdots\,d_t^{e_t})\in \mathcal S(n)$ if it factors 
as $f(x) = \prod_{j=1}^t f_j(x)^{e_j}$, where the
$f_j$ are distinct irreducible monic polynomials  over~$\F_p$ with $\deg(f_j) = d_j$.
We write $\sigma(f)$ for the splitting type of $f$, and 
$N_\sigma$ for the number of monic polynomials in $\F_p[x]$
with splitting type $\sigma$.

If $\sigma = (d)$,  then we simply write $N_d$ for $N_\sigma$. That is, $N_d$ is 
the number of degree $d$ irreducible monic polynomials in $\F_p[x]$.
Writing $\mu$ for the M\"obius function, it is well known that
\[ N_d = \frac{1}{d} \sum_{k|d} \mu(k) p^{d/k}. \]
In general, for $\sigma = (d_1^{e_1}\,d_2^{e_2}\,\cdots\,d_t^{e_t})
\in \stypes(n)$, we have 
\begin{equation}\label{numberofpolynomialsoffixedsplittingtype}
N_{\sigma} = \prod_{d=1}^n \binom{N_d}{m_d}
\binom{m_d}{m_{d1}\ m_{d2}\ \cdots\ m_{dn}},
\end{equation}
where  
\[ m_{de} = m_{de}(\sigma) := \# \{ s : d_s^{e_s} = d^e \}, \]
and
\[ m_d = m_{d}(\sigma) := \# \{ s : d_s = d \} = \sum_{e=1}^n m_{de}. \]
Since there are~$p^n$ monic polynomials of degree~$n$ in $\F_p[x]$,
the probability that a degree $n$ monic polynomial
$f \in \F_p[x]$ has splitting type~$\sigma$, for $\sigma\in \stypes(n)$, is $N_{\sigma}/p^n$.
This is evidently a rational function of $p$.

\subsection{Power series identities involving \texorpdfstring{$N_\sigma$}{Nσ}}

We now establish some power series identities involving the 
counts $N_\sigma$ defined in the previous section.

Let $x_{de}$ for~$d,e\ge1$ be indeterminates. For a splitting type $\sigma\in\mathcal S(n)$
of degree $n$, let $$x_{\sigma}=\prod_{d^e\in\sigma}x_{de}.$$
Polynomials in the $x_{de}$
will be weighted by setting $\wt(x_{de})=de$.  
We set $y_0 = 1$,  and for $n \ge 1$ define
\[
y_n = \sum_{\sigma\in\stypes(n)}N_{\sigma}x_{\sigma},
\]
so that every monomial in~$y_n$ has weight~$n$. We
set $x_{d0} = 1$ for all $d \ge 1$.
\begin{proposition}
  \label{prop:general-identity}
 We have the following identity in $\Z[\{ x_{de}\}_{d,e\geq 1}][[t]]$:
\begin{equation}
  \label{eqn:general-identity}
\sum_{n=0}^\infty y_nt^n = \prod_{d=1}^\infty 
  \left(\sum_{e=0}^\infty x_{de}t^{de}\right)^{N_d}.
\end{equation}
\end{proposition}
\begin{proof}
We must show that when the right hand side is multiplied out, the
coefficient of $t^n$ is $y_n$.  The coefficient of $t^n$ is a sum of
monomials in the~$x_{de}$ of weight~$n$.  Each such product has the
form $x_\sigma$ for some~$\sigma\in\stypes(n)$, and the number of times
each monomial occurs is~$N_{\sigma}$.
\end{proof}

By specializing the $x_{de}$, we obtain the following corollary.
\begin{corollary}
\label{cor:special1} We have the following identity in $\Z[[t]]$: 
\begin{equation}
  \label{eqn:identity1}
(1-pt)^{-1} = \prod_{d=1}^\infty (1-t^d)^{-N_d}.
\end{equation}
\end{corollary}
\begin{proof}
In~\eqref{eqn:general-identity}, set $x_{de}=1$ for all~$d,e$. Then
$x_{\sigma}=1$, so $y_n=p^n$, and~\eqref{eqn:identity1} follows.
\end{proof}

\begin{corollary}
Let $x_e$ for $e \ge 1$ be indeterminates, and set $x_{0}=1$. Then, in $\Z[x_1,x_2,\ldots][[t]]$, we have:
\label{cor:special2}
\begin{equation}
  \label{eqn:identity2}
\sum_{n=0}^\infty \sum_{\sigma \in \stypes(n)} N_{\sigma} 
 \left(\prod_{1^e \in \sigma} x_e\right) t^n
= \left(\sum_{n=0}^\infty x_n t^n \right )^{p}
  ( 1 - t )^p ( 1 - pt)^{-1}.
\end{equation}
\end{corollary}
\begin{proof}
In~\eqref{eqn:general-identity}, set $x_{1e}=x_e$, and set $x_{de}=1$ for all~$d\ge2$.  Then, by Corollary~\ref{cor:special1},
we have
\[
\prod_{d=2}^\infty (1-t^d)^{-N_d} = (1-t)^p(1-pt)^{-1},
\]
yielding~\eqref{eqn:identity2}.
\end{proof}

\subsection{Resultants, coprime factorizations, and independence}

\subsubsection{Resultants}
We begin with an observation about resultants of polynomials in $\Zp[x]$
and their behavior upon reduction modulo~$p$.

\begin{lemma}\label{lem:resultant-reduction}
Let $f,g\in\Zp[x]$ have degrees~$m$ and~$n$ respectively.
\begin{enumerate}
\item If the leading coefficients of $f$ and~$g$ are both units, then
  $\overline{\Res(f,g)}=\Res(\fbar,\gbar)$.
\item If the leading coefficient~$a_m$ of $f$ is a unit and
  $d=\deg(\gbar)<n$, then $\overline{\Res(f,g)} =
  \overline{a_m}^{n-d}\Res(\fbar,\gbar)$.
\item If the leading coefficients of $f$ and $g$ are both non-units, then
  $\overline{\Res(f,g)}=0$.
\end{enumerate}
\end{lemma}

\begin{proof}
  These are standard properties of resultants and may be seen by
  examination of the definition of~$\Res(f,g)$ as the value of the
  $(m+n)\times(m+n)$ Sylvester determinant.
\end{proof}
\begin{corollary}\label{cor:unit-resultant}
  Let $f,g\in\Zp[x]$ have degrees~$m$ and~$n$ respectively.
  Then $\Res(f,g)$ is a unit if and only if at least one of the
  leading coefficients of $f,g$ is a unit, and the reductions
  $\fbar,\gbar$ are coprime.
\end{corollary}

Our reason to consider resultants is the following. 

\begin{lemma}\label{lem:jac}
Let $R$ be a ring. 
For any $d\ge1$, we identify $R[x]_d^1\cong R^{d}$ and $R[x]_d\cong R^{d+1}$ as $R$-modules.
\begin{itemize}

\item[{\rm (a)}]
The multiplication map $R[x]_m^1\times R[x]_n^1 \to R[x]_{m+n}^1$ 
has Jacobian given by~$\Res(f,g)$.
 
\item[{\rm (b)}]
The multiplication map $R[x]_m^1\times R[x]_n \to R[x]_{m+n}$ 
has Jacobian given by~$\Res(f,g)$.
\end{itemize}
\end{lemma}

\begin{proof}
{We first consider case (a), when both polynomials are monic.} Let 
$f(x)=x^m+\sum_{i=0}^{m-1}a_i x^i$, $g(x)=x^n+\sum_{j=0}^{n-1}b_j x^j$, and  $h(x)=x^{m+n}+\sum_{k=0}^{m+n-1}c_k x^k$ 
be monic polynomials in $R[x]$ 
having degrees $m$, $n$, and $m+n$ respectively.  If~$h(x)=f(x)g(x)$, then $c_k=\sum_{i+j=k}a_i b_j$, and the matrix of partial derivatives of the $c_k$ with respect to the $a_i$ and $b_j$ is precisely the Sylvester matrix whose determinant is $\Res(f,g)$. 

We next consider case (b), and assume that  $f(x)=x^m+\sum_{i=0}^{m-1}a_ix^i\in R[x]_m^1$ is monic while $g(x)=\sum_{j=0}^{n}b_jx^j\in R[x]_n$ is not necessarily so.  Let $f(x)g(x)=\sum_{k=0}^{m+n}c_kx^k$, and let $M$ be the $(m+n+1)\times (m+n+1)$ matrix of partial derivatives of the $c_k$ with respect to the $a_i$ and $b_j$. Since $c_{m+n}=b_n$, the last row consists of 0's except for the final entry which is 1. Expanding the determinant by the last row, we again obtain  $\Res(f,g)$.  \end{proof}

\begin{corollary}
\label{lem:mpbj}
Let $A\subset \Z_p[x]_m^1$, $B\subset \Z_p[x]_n^1$  $($resp.\  $B\subset \Z_p[x]_n)$, and $AB\subset \Z_p[x]_{m+n}^1$ $($resp.\  $AB\subset \Z_p[x]_{m+n})$
be measurable subsets such that  multiplication
induces a bijection
\[
A\times B \to AB = \{ab\mid a\in A, \:b\in B\}.
\]
If $\Res(a,b)\in\Zp^*\,$ for all~$a\in A$ and $b\in B$, then this bijection is
measure-preserving.
\end{corollary}

\subsubsection{Coprime factorizations and Hensel lifting}
We next recall Hensel's lemma for polynomial factorizations in certain 
quantitative forms.  The first is standard, and is stated as Lemma~2.3
in~\cite{Buhler-et-al}, while the variant is mentioned
in~\cite[p.~24]{Buhler-et-al}.

For $f\in\Fp[x]_d^1$, we denote by
$P_f$ the set of polynomials in $\Zp[x]_d^1$ that reduce to~$f$
modulo~$p$; 
and for $n\ge d$, we denote by $P_f^n$ the set of polynomials in $\Zp[x]_n$
that reduce to~$f$ modulo~$p$.  

\begin{lemma}\label{lem:poly-hensel}
Suppose that $g,h\in \Fp[x]$ are monic and coprime. Then the 
multiplication map 
\begin{equation}  \label{eqn:mp}
P_g \times P_h \to P_{gh} \end{equation} is a measure-preserving 
bijection.
\end{lemma}
\begin{proof}
Let $f\in\Zp[x]_n^1$ be such that $\fbar$ factors in~$\Fp[x]$ as
$\fbar= g h$. Then by Hensel's lemma $f$ factors uniquely in $\Z_p[x]$ 
as $f=\tilde g \tilde h$, where $\tilde g \in P_{g}$ and $\tilde h \in P_{h}$. 
Therefore~\eqref{eqn:mp} is a bijection. The measure-preserving property 
holds by Corollaries~\ref{cor:unit-resultant} and~\ref{lem:mpbj}.
\end{proof}

The following variants will be used to handle polynomials 
$f \in \Zp[x]$ whose leading coefficient is not a unit.

\begin{lemma}\label{lem:poly-hensel-0}
Let $n \ge m \ge 0$, and consider the multiplication map
\begin{equation}\label{eqn:mp0}
 \mu:\  \Z_p[x]^1_m
 \times
 P_1^{n-m} 
 \to \{ f \in \Zp[x]_n : \fbar \in \Fp[x]^1_m \}.
\end{equation}
\begin{itemize}
\item[{\rm (a)}] $\mu$ is a measure-preserving bijection.
\item[{\rm (b)}] Let $r$ be an integer satisfying $m\le r\le n$.  In
  \eqref{eqn:mp0}, replace the set on the right-hand side with the
  subset of $f\in \Zp[x]_n$ also satisfying $p^{r}f(x/p)\equiv
  x^{n}\pmod{p}$, and replace the second factor on the left-hand side
  with the subset of $h\in P_1^{n-m}$ satisfying $p^{r-m}h(x/p)\equiv
  x^{n-m}\pmod{p}$.  Then the restriction of~$\mu$ to these subsets is
  still a measure-preserving bijection.
\end{itemize}
\end{lemma}

\begin{proof}
(a) Let $f\in\Zp[x]_n$ be such that $\fbar$ is monic of degree $m$.
  Then homogenising, applying Hensel's lemma, and dehomogenising,
  shows that $f$ factors uniquely in $\Zp[x]$ as $f=gh$ where $g \in
  \Z_p[x]^1_m$ and $h \in P_1^{n-m}$.  Therefore, \eqref{eqn:mp0} is a
  bijection. The measure-preserving property again holds by
  Corollaries~\ref{cor:unit-resultant} and~\ref{lem:mpbj}, since $g$
  is monic.

  (b) Suppose that $f\in\Zp[x]_n$ is such that $\fbar$ is monic of
  degree $m$, and also that $p^rf(x/p)\equiv x^n\pmod{p}$; by this we
  mean that $p^rf(x/p)\in\Z_p[x]$ and $p^rf(x/p)-x^n\in p\Z_p[x]$.
  Factoring $f=gh$ as before, we have
  \[
  p^m g(x/p) \cdot p^{r-m} h(x/p) = p^r f(x/p) \equiv x^n \pmod{p};
  \]
  since $g\in\Z_p[x]_m^1$, we have $p^m g(x/p)\in\Z_p[x]$ and $p^m
  g(x/p)\equiv x^m\pmod{p}$.  Hence, by Gauss's Lemma, $p^{r-m}
  h(x/p)\in\Z_p[x]$, and (using unique factorization in $\F_p[x]$)
  also $p^{r-m} h(x/p) \equiv x^{n-m}\pmod{p}$.  Conversely if $h$
  satisfies $p^{r-m} h(x/p) \equiv x^{n-m}\pmod{p}$, then since $p^m
  g(x/p)\equiv x^m\pmod{p}$ for all $g\in\Z_p[x]_m^1$, it follows that
  $f=gh$ satisfies  $p^rf(x/p)\equiv x^n\pmod{p}$.

  Thus $\mu$ restricts to a bijection between the subsets on each
  side, and is measure-preserving, as before.
\end{proof}

\subsubsection{Independence lemmas}

Finally, we may phrase
Lemmas~\ref{lem:poly-hensel}
and \ref{lem:poly-hensel-0} as statements regarding the  independence of   suitable random variables. 

\begin{corollary} 
\label{cor:indep}
Let $g,h\in\Fp[x]$ be coprime monic polynomials.
For $f\in P_{gh}$, let $\pi_1$ and $\pi_2$ denote the projections of $P_{gh}$ onto $P_g$ and $P_h$, respectively,  under the bijection $P_{gh}\to P_g\times P_h$.  
 Then the number of $\Qp$-roots of 
$f \in P_{gh}$ is $X + Y$,  where $X,Y:P_{gh}\to \{0,1,2, \ldots
\}$ are   independent~random variables  distributed on $f\in P_{gh}$ as the number
of $\Qp$-roots of $\pi_1(f)\in P_g$ and $\pi_2(f)\in P_h$, respectively. 
\end{corollary}

\begin{corollary} \label{cor:stability}
Let $m \le n$,  and let $$B_{m,n} := \{ f \in \Zp[x]_n : \fbar \in \Fp[x]^1_m \}.$$
For $f\in B_{m,n},$ let $\psi_1$ and $\psi_2$ denote the projections of $B_{m,n}$ onto $\Z_p[x]_m^1$ and $P_1^{n-m}$, respectively,  under the bijection $B_{m,n}\to \Z_p[x]_m^1\times P_1^{n-m}$. Let
$X,Y : B_{m,n} \to \{0,1,2,\ldots\}$ be the random variables
giving the numbers of roots of $f \in B_{m,n}$ in $\Z_p$ and in $\Qp \setminus \Zp$,   respectively. Then $X$ and $Y$ are independent random 
variables  distributed on $f\in B_{m,n}$ as the number of $\Qp$-roots
of $\psi_1(f)(x)\in \Z_p[x]_m^1$ and of
$\psi_2(f)^\rev(x):=x^{n-m}\psi_2(f)(1/x)\in P_{x^{n-m}}$,
respectively.
\end{corollary}

\section{Proof of Theorem~\ref{thm:totsplit}}\label{sec:proofof12}

\subsection{Theorem~\ref{thm:totsplit}(b) implies 
Theorem~\ref{thm:totsplit}(a)}
\label{b=>a}
Theorem~\ref{thm:totsplit}(b) allows us to compute $\alpha(n,d)$, $\beta(n,d)$,
and $\rho(n,d)$ for any $n$ and $d$. Indeed we use~\eqref{alphatobeta:mom} 
and \eqref{betatoalpha:mom} to solve for the $\alpha$'s 
and $\beta$'s, and then~\eqref{rho:mom} to compute the $\rho$'s. 
The answers obtained are rational functions of $p$. 
The relation~\eqref{alphatobeta:mom} is invariant under replacing 
$t\rightarrow t/p$ and switching 
$p \leftrightarrow 1/p$ and $\genA_d \leftrightarrow \genB_d$, while 
the relation~\eqref{betatoalpha:mom}
is invariant under switching 
$p \leftrightarrow 1/p$ and $\genA_d \leftrightarrow \genB_d$. 
The symmetry~\eqref{fe:alphabeta}
then follows by induction on $n$ and $d$, while~\eqref{fe:rho} 
follows from~\eqref{rho:mom}. 
\qed

\subsection{Proof of Theorem~\ref{thm:totsplit}(b)}

\subsubsection{Conditional expectations}

The expectations $\alpha(n,d)$ and $\beta(n,d)$ were defined in the
introduction. To help evaluate them, we make the following
additional definitions.

\begin{definition}
\label{def:cond}
\phantom{hello}

\begin{itemize}
\item[{\rm (i)}] For $f\in\F_p[x]_n^1$, let $\alpha(n,d \mid f)$ denote the 
  expected number of $d$-sets of $\Qp$-roots of a polynomial in 
  $P_f \subset \Zp[x]_n^1$. Since $P_{f}$ has relative 
  density $p^{-n}$ in $\Zp[x]_n^1$, we have 
\begin{equation} \label{eqn:alpha-f}
  \alpha(n,d) = p^{-n} \sum_{f\in\Fp[x]_n^1} \alpha(n,d \mid f).
\end{equation}
Also,  $\beta(n,d)=\alpha(n,d\mid x^n)$.  
\item[{\rm (ii)}] For $\sigma$ in $\stypes(n)$, let $\alpha(n,d \mid \sigma)$ be
  the expected number of $d$-sets of $\Qp$-roots of a polynomial
  in~$\Zp[x]_n^1$ whose mod $p$ splitting type is $\sigma$.  Thus
\begin{equation} \label{eqn:alpha-sigma}
  \alpha(n,d) = p^{-n} \sum_{\sigma\in\stypes(n)} N_{\sigma} \,
  \alpha(n,d \mid \sigma),
\end{equation}
and
\begin{equation} \label{eqn:alpha-sigma-f}
\alpha(n,d \mid \sigma) = 
N_{\sigma}^{-1} \sum_{f\in\Fp[x]_n^1:\ \sigma(f)=\sigma}\alpha(n,d\mid f),
\end{equation}
where $\sigma(f)$ denotes the splitting type of $f$.
\end{itemize}
\end{definition}

\subsubsection{Writing the \texorpdfstring{$\alpha$}{α}'s in terms of the \texorpdfstring{$\beta$}{β}'s}
\label{sec:atob}

The aim of this subsection is to prove~\eqref{alphatobeta:mom}, the first
part of Theorem~\ref{thm:totsplit}(b).

\begin{lemma}
\label{lem:exps}
Let $g, h \in\Fp[x]$ be monic and coprime.  Then
\begin{equation}
\label{nostars}
 \alpha(\deg(gh),d \mid  g h) = \sum_{d_1+d_2=d}
   \alpha(\deg(g),d_1 \mid g)\cdot \alpha(\deg(h),d_2 \mid h),
\end{equation}
where the sum is over all pairs $(d_1,d_2)$ of non-negative integers
summing to $d$.

If, additionally, $h$ has no roots in~$\Fp$, then
\[
\alpha(\deg(gh),d \mid g h) = \alpha(\deg(g),d \mid g).
\]

\end{lemma}
\begin{proof}
The lemma follows from Corollary~\ref{cor:indep} and the observation
that if $X$ and $Y$ are independent random variables taking values in
$\{0,1,2, \ldots \}$ then 
\begin{equation} \label{eq:exp-sum}
\E \binom{X+Y}{d} \,=\, \sum_{d_1+d_2 = d}
\E \binom{X}{d_1} \E \binom{Y}{d_2}. \qedhere
\end{equation}
\end{proof}

Recall that $\beta(n,d) = \alpha(n,d \mid x^n)$
is the expected number of $d$-sets of roots of a monic polynomial of
degree~$n$ which reduces to~$x^n$ modulo~$p$. Using Lemma~\ref{lem:exps}, we can express $\alpha(n,d \mid f)$ for
monic~$f\in\Fp[x]_n$ in terms of $\beta(n',d')$ for
appropriate~$n',d'$.  

\begin{lemma}
\label{lem:alphasfrombetas}
Let $\sigma = (1^{n_1} \cdots 1^{n_k} \cdots) \in\stypes(n)$ be a
splitting type with exactly~$k=m_1(\sigma)$ powers of $\,1$. 
Then
\begin{equation}\label{eqn:sigmaeq} \alpha(n,d \mid \sigma)
= \sum_{d_1 + \cdots + d_k = d} \,\, \prod_{i=1}^k \beta(n_i,d_i).
\end{equation}
\end{lemma}
\begin{proof}
Let $f\in\Fp[x]_n^1$ have splitting type $\sigma$. To evaluate $\alpha(n,d \mid f)$, we may ignore the factors of~$f$ of degree greater than~$1$, since if
$f=f_1f_2$ where $\sigma(f_1)=(1^{n_1} \cdots 1^{n_k})$ and $f_2$
has no linear factors, then $\alpha(n,d \mid f) = \alpha(\deg(f_1), d
\mid f_1)$ by the last part of Lemma~\ref{lem:exps}.

Now let $f=\prod_{i=1}^k\ell_i^{n_i}$, where the $\ell_i$ are distinct,
monic, and of degree~$1$.  Using Lemma~\ref{lem:exps} repeatedly gives
\[
\alpha(n,d \mid f) = \sum_{d_1 + \cdots + d_k = d} \,\, \prod_{i=1}^k
\alpha(n_i,d_i \mid \ell_i^{n_i}).
\]
Finally, $\alpha(n_i,d_i \mid \ell_i^{n_i}) = \alpha(n_i,d_i \mid
x^{n_i}) = \beta(n_i,d_i)$, since for fixed $c \in \Z_p$ the map
$g(x) \mapsto g(x+c)$ is measure-preserving on monic polynomials
in $\Zp[x]$ of a given degree.
Thus
\begin{equation}\label{eqn:feq} \alpha(n,d \mid f)
= \sum_{d_1 + \cdots + d_k = d} \,\, \prod_{i=1}^k \beta(n_i,d_i),
\end{equation}
and (\ref{eqn:sigmaeq}) now follows  from~\eqref{eqn:alpha-sigma-f}
and \eqref{eqn:feq}.
\end{proof}

\begin{proof}[Proof of Theorem~$\ref{thm:totsplit}(b)$, 
Equation \eqref{alphatobeta:mom}]
Let $\sigma=(1^{n_1}\cdots1^{n_k}\cdots)\in\mathcal S(n)$ be as in 
Lemma~\ref{lem:alphasfrombetas}. Then, 
by~\eqref{eqn:alpha-sigma} and Lemma~\ref{lem:alphasfrombetas}, 
we have
\begin{equation}
\label{alphasfrombetas}
\alpha(n,d) = p^{-n} \sum_{\sigma \in \stypes(n)}
N_{\sigma} \, \alpha(n,d \mid \sigma) = p^{-n} \sum_{\sigma \in \stypes(n)}
N_{\sigma} 
\sum_{d_1 + \cdots + d_k = d} \,\, \prod_{i=1}^k \beta(n_i,d_i).
\end{equation}
Multiplying by $u^d$ and summing over $d$ gives 
\[ \sum_{d=0}^n  \alpha(n,d) u^d = 
   p^{-n} \sum_{\sigma \in \stypes(n)} 
N_{\sigma} \prod_{1^e \in \sigma} 
\left( \sum_{d=0}^e \beta(e,d) u^d \right). \]
Multiplying by $(pt)^n$, summing over $n$, and using 
Corollary~\ref{cor:special2}, we obtain 
\[ \sum_{d=0}^\infty \left( \sum_{n=0}^\infty \alpha(n,d) (pt)^n 
\right) u^d = \left( \sum_{d=0}^\infty \left( \sum_{n=0}^\infty
\beta(n,d) t^n \right) u^d \right)^p 
(1 - t)^p ( 1 - pt)^{-1}. \]
Finally, multiplying both sides by $1-pt$ yields~\eqref{alphatobeta:mom}.
\end{proof}

\subsubsection{Writing the \texorpdfstring{$\rho$}{ρ}'s in terms 
of the \texorpdfstring{$\alpha$}{α}'s and \texorpdfstring{$\beta$}{β}'s}
\label{sec:abtorho}
The aim of this section is to prove~\eqref{rho:mom}, the second part
of Theorem~\ref{thm:totsplit}(b).

Recall that $\rho(n,d)$ is the expected number of $d$-sets of
$\Qp$-roots of polynomials~$f\in\Zp[x]$ of degree~$n$.
It is evident that this does not change if we restrict to 
primitive polynomials.

Let $f\in\Zp[x]$ be a primitive polynomial of degree~$n$.  
Let $m = \deg(\fbar)$ be the reduced degree of~$f$.
For fixed~$m$ with $0\le m\le n$, the density of primitive polynomials~$f\in\Zp[x]_n$ with
reduced degree~$m$ is $\frac{p-1}{p^{n+1}-1}p^m$.  Therefore, conditioning 
on the value of~$m$, we have
\begin{equation}
\label{getrho-new}
\rho(n,d) = \frac{p-1}{p^{n+1}-1} \sum_{m=0}^n p^m  \rho(n,d, m),
\end{equation}
where $\rho(n,d, m)$ is the expected number of $d$-sets of $\Qp$-roots
of~$f$ as $f\in \Zp[x]_n$ runs over polynomials of degree~$n$ with  reduced degree~$m$.
This expectation does not change if we restrict to $f$ whose reduction 
mod $p$ is monic.

Equation~\eqref{rho:mom} now follows   from~\eqref{getrho-new} and 
the following lemma. 

\begin{lemma}
\label{lem:rho-cond}
We have
\begin{equation}
\label{eqn:rho-cond}
\rho(n,d, m) = \sum_{d_1 + d_2 = d} \alpha(m,d_1)\cdot \beta(n-m,d_2).
\end{equation}
\end{lemma}
\begin{proof}
This follows from Corollary~\ref{cor:stability} and \eqref{eq:exp-sum}.
\end{proof}

\subsubsection{Writing the \texorpdfstring{$\beta$}{β}'s in terms of the \texorpdfstring{$\alpha$}{α}'s}
\label{sec:btoa}

The aim of this section is to prove~\eqref{betatoalpha:mom}, the third
and last part of Theorem~\ref{thm:totsplit}(b).

Fixing $d$, we put $\alpha_n := \alpha(n,d)$ and $\beta_n := \beta(n,d)$.
In the following lemma, we express  $\beta_n$ in terms of $\alpha_s$ for $s\le n$.

\begin{lemma}
\label{lem:btoa}
We have
\begin{equation}
\label{betasfromalphas}
\beta_n =  p^{-\binom{n}{2}}\alpha_n+(p-1) \sum_{0\le s<r<n}p^{-\binom{r+1}{2}}p^s\alpha_s.
\end{equation}
\end{lemma}

\begin{proof}
  Recall that
  $\beta_n$ is the expected 
  number of $d$-sets of $\Zp$-roots of~$f\in P_{x^n}$.  
  All such roots must lie in $p \Z_p$, and thus  correspond to 
  $\Zp$-roots of $f(px)$. To each $f\in
  P_{x^n}$, we associate a pair of integers $(r,s)$ with $0\le s\le r\le
  n$ as follows. Consider $f(px)$, and let $r$ be the largest integer
  such that $p^r\mid f(px)$, so that $1\le r\le n$. Let $s$ be the
  reduced degree of~$g(x)=p^{-r}f(px)$.  Then either $0\le s<r<n$, or
  $s=r=n$.

  The relative density of the subset of $f\in P_{x^n}$ such that
  $p^r\mid f(px)$ is $p^{-\binom{r}{2}}$, since for $0\le i\le r-2$ we
  require the coefficient of~$x^i$ in~$f$ to be divisible by~$p^{r-i}$
  and not just by~$p$.  Given $r<n$, the condition that $g(x)$
  has reduced degree at least~$s$ imposes $r-s-1$ additional
  divisibility conditions, so the relative density of those~$f$ such that
  the reduced degree is exactly~$s$ is $p^{-(r-s-1)}(1-1/p) =
  p^{s-r}(p-1)$.  Thus  the relative density of $f\in P_{x^n}$ with
  parameters $(r,s)$ is given by $p^{-\binom{r}{2}}p^{s-r}(p-1) =
  p^{-\binom{r+1}{2}}p^s(p-1)$ for $0\le s<r<n$.  If $r=n$, then
  $s=r$, and therefore the density of~$f$ with parameters~$(n,n)$ is
  $p^{-\binom{n}{2}}$.

  If $s=r=n$, then $g$ is distributed as an arbitrary element of
  $\Z_p[x]_n^1$, while if $s<r<n$ then $g\in\Z_p[x]_n$ is subject to
  the conditions that $\overline{g}$ has degree~$s$, and that $p^r
  g(x/p) = f(x) \equiv x^n\pmod{p}$.  Hence in both cases, given the
  values of $r$ and~$s$, the conditional expected number of $d$-sets
  of $\Zp$-roots of~$f\in P_{x^n}$ is $\alpha_s$ (independent of $r$);
  in the case $s<r<n$, this follows from
  Lemma~\ref{lem:poly-hensel-0}(b) by considering the restriction of
  the random variable~$X$ in Corollary~\ref{cor:stability} to the
  appropriate subset.  Hence $\beta_n =
  p^{-\binom{n}{2}}\alpha_n+\sum_{0\le
    s<r<n}p^{-\binom{r+1}{2}}p^s(p-1)\alpha_s.$
\end{proof}

\begin{proof}[Proof of \eqref{betatoalpha:mom}]
Taking Equation~\eqref{betasfromalphas} for $n$ and~$n-1$
and subtracting gives
\begin{equation}\label{eqn:subt}
p^{\binom{n}{2}}(\beta_n-\beta_{n-1}) =
(\alpha_n-p^{n-1}\alpha_{n-1}) + (p-1)\sum_{s=0}^{n-2}p^s\alpha_s.
\end{equation}
Now taking Equation (\ref{eqn:subt}) for $n$ and $n-1$ and  again  subtracting yields
\[
p^{\binom{n}{2}}[(\beta_n-\beta_{n-1})  -p^{1-n}(\beta_{n-1}-\beta_{n-2})] =
(\alpha_n-\alpha_{n-1}) - p^{n-1}(\alpha_{n-1}-\alpha_{n-2}),
\]
and this indeed asserts the equality of the coefficient of $t^n$ on both sides
of~\eqref{betatoalpha:mom}.
\end{proof}

We have completed the proof of Theorem~\ref{thm:totsplit}(b).

\begin{remark}
Equations~\eqref{alphasfrombetas}, \eqref{getrho-new}, \eqref{eqn:rho-cond}
and \eqref{betasfromalphas} are sufficient to compute the $\alpha$'s,
$\beta$'s and $\rho$'s.
We~were~motivated to find the neater formulation
in Theorem~\ref{thm:totsplit}(b) by the desire to prove the $p \leftrightarrow
1/p$ symmetries.
\end{remark}

\subsection{Proof of Theorem~\ref{thm:totsplit}(c)}
\label{sec:pfthm3}

Consider a random polynomial of degree $n$ in $\Zp[x]$.
Let $\widetilde{\alpha}(n,d)$ be the expected number of
$d$-sets of roots in $\Zp$. Conditioning on the reduced degree and applying  Corollary~\ref{cor:stability} shows that
\[ \widetilde{\alpha}(n,d) = \sum_{m=0}^n \left(1 - \frac{1}{p} \right) 
\frac{1}{p^m} \alpha(n-m,d) + \frac{1}{p^{n+1}} \widetilde{\alpha}(n,d). \]
This rearranges to give
\begin{equation}
\label{a:cond}
\widetilde{\alpha}(n,d) = \sum_{m=0}^n (1 - p) p^m \alpha(m,d) + p^{n+1} 
\widetilde{\alpha}(n,d).
\end{equation}
In other words, $\widetilde{\alpha}(n,d)$ is a weighted average of the $\alpha(m,d)$
for $m \leq n$.

We now show that $\alpha(n,d)$ and $\widetilde{\alpha}(n,d)$ are equal and independent of $n$, provided that
$n \ge 2d$.

Let $A_n=\Zp[X]_n^1$ denote the set of monic polynomials over
$\Zp$ of degree~$n$, and $B_n$ the set of all polynomials of degree less than
$n$. Then we have $A_n = \{X^n+h: h \in B_n\}$, and both $A_n$ and~$B_n$ may be identified
with $\Zp^n$ and have measure~$1$.  Let $A_n^\spl$ be the subset of
those $f$ in $A_n$ that split completely.  The measure of $A_n^\spl$
is $\alpha(n,n)$.  

Now consider the multiplication map $A_d^\spl \times \Zp[x]_{n-d} \to \Zp[x]_{n}$,
whose image is the set of $f\in\Zp[x]_{n}$ with at least $d$ roots in
$\Zp$; in general, the number of preimages of $f$ in $\Zp[x]_{n}$ is equal to the
number of $d$-sets of roots of $f$ in $\Zp$. 
This implies that $\widetilde{\alpha}(n,d)$ is the $p$-adic measure
of the image of the multiplication map, viewed as a multiset. The change of variables from $A_d^\spl \times \Zp[x]_{n-d}$ 
to $\Zp[x]_{n}$ introduces a Jacobian factor which, 
by Lemma~\ref{lem:jac}, is just the resultant. Therefore, 
\begin{equation}
\label{doubleintegraltilde}
 \widetilde{\alpha}(n,d) = \int_{g \in A_d^\spl} \int_{h \in \Zp[x]_{n-d}}
 |\Res(g,h)|\; dh\; dg.
\end{equation}
 Similarly, we have
\begin{equation}
\label{doubleintegral}
 \alpha(n,d) = \int_{g \in A_d^\spl} \int_{h \in A_{n-d}}
 |\Res(g,h)|\; dh\; dg.
\end{equation}
The following lemma now proves the first part of
Theorem~\ref{thm:totsplit}(c), namely, that $\genA_d(t)$ is a polynomial of degree
at most $2d$. 

\begin{lemma}
\label{indep-lem1}
The expectations $\alpha(n,d)$ and $\widetilde{\alpha}(n,d)$ are equal and independent of $n$ for $n \ge 2d$.
\end{lemma}
\begin{proof}
By~\eqref{doubleintegraltilde} and ~\eqref{doubleintegral} it suffices to show that for each fixed $g$
in $A_d^{\spl}$, the values of the inner integrals $\int_{h \in \Zp[x]_{n-d}}
 |\Res(g,h)|\; dh$ and $\int_{h \in A_{n-d}}
|\Res(g,h)| dh$ are equal and independent of $n$ for $n \ge 2d$.  Our argument is
quite general, in that we only use that $g$ is monic, not that it is
split.

We assume that $n \ge 2d$, and write each $h\in\Zp[x]_{n-d}$ uniquely as
$h=qg+r$ with $q\in\Zp[x]_{n-2d}$ and $r\in B_d$.  This sets up a
bijection $(q,r) \mapsto h=qg+r$ from $\Zp[x]_{n-2d} \times B_d$ to
$\Zp[x]_{n-d}$ (using here that $n-d \ge d$). Now using $\Res(g,h)=\Res(g,r)$, and the fact that our bijection has trivial Jacobian (the change of basis matrix is triangular with 1's on the
diagonal since $g$ is monic), we deduce that
\[ \int_{h \in \Zp[x]_{n-d}} |\Res(g,h)| dh = \int_{q \in \Zp[x]_{n-2d}} \int_{r \in
B_d} |\Res(g,r)| dr dq = \int_{r \in B_d} |\Res(g,r)| dr, \] since the
integral over~$q\in \Zp[x]_{n-2d}$ is just the measure of $\Zp[x]_{n-2d}$  
which
is~$1$. In an identical manner, we have
\[ \int_{h \in A_{n-d}} |\Res(g,h)| dh = \int_{q \in A_{n-2d}} \int_{r \in
B_d} |\Res(g,r)| dr dq =\int_{r \in B_d} |\Res(g,r)| dr. \]
Hence
\[
\widetilde{\alpha}(n,d)
   = \alpha(n,d) = 
   \int_{g \in A_d^\spl} \int_{r \in B_d} |\Res(g,r)|\; dr\; dg 
   \]
for $n\geq 2d$. 
The inner integral above  clearly depends on $g$ and $d$, but not on
$n$. 
\end{proof}

We now turn to proving the remaining parts of Theorem~\ref{thm:totsplit}(c).
By Lemma~\ref{indep-lem1}, we have that $\genA_d(t)$ is a polynomial
of degree at most $2d$. Thus, fixing any $n \ge 2d$, we may write
\begin{equation}
\label{writeaspoly}
 \genA_d(t) = (1-t) \sum_{m=0}^n \alpha(m,d) t^m + \alpha(n,d)t^{n+1}. 
\end{equation}
Lemma~\ref{indep-lem1} allows us to replace $\widetilde{\alpha}(n,d)$
by $\alpha(n,d)$ in~\eqref{a:cond}.
Taking $t=1$ in~\eqref{writeaspoly} shows that
the left hand side of~\eqref{a:cond} is
$\genA_d(1)$. Taking $t=p$ in~\eqref{writeaspoly} shows that the right hand side 
of~\eqref{a:cond} is $\genA_d(p)$. Therefore, 
$\genA_d(1) = \genA_d(p)$. 

Since $\genA_d$ is a polynomial of
  degree at most $2d$, it follows by~\eqref{alphatobeta:mom}, or
  equally~\eqref{fe:alphabeta}, that $\genB_d$ is a polynomial of
  degree at most $2d$. Directly from the definitions of $\genA_d$ and $\genB_d$, these results are equivalent to the statements that $\alpha(n,d) = \genA_d(1)$ and $\beta(n,d) = \genB_d(1)$ for all $n \geq 2d$. 
  
  It follows by~\eqref{rho:mom} that $\genR_d$ is
  a polynomial of degree at most $2d$.  
  To prove the stabilisation result for the $\rho(n,d)$,  
  we use the fact we just proved that 
  $\genA_d(1) = \genA_d(p)$.  It
  follows by~\eqref{alphatobeta:mom}, or equally~\eqref{fe:alphabeta},
  that $\genB_d(1) = \genB_d(1/p)$.   By~\eqref{rho:mom}, we then have
  $\genR_d(1) = \genR_d(1/p)$.  We may therefore write
  $\genR_d(t) = \genR_d(1) + (1-t)(1-pt) F(t)$ where $F$ has degree at
  most $2d - 2$. Finally, from the definition of $\genR_d$,  we have
  $\rho(n,d) = \genR_d(1)$ for all $n > \deg(F)$.

This completes the proof of Theorem~\ref{thm:totsplit}(c).

\begin{remark}
The values of the $\widetilde\alpha(n, d)$, which may be computed from the $\alpha(n, d)$ using (\ref{a:cond}), may also be of independent interest. For example, the expectation $\widetilde\alpha(n, 1) = p/(p + 1)$ is computed by
Caruso~\cite{Caruso}, and also follows from the one-variable case of the work of Evans~\cite[Theorem~1.2]{Evans}.
\end{remark}

\section{Asymptotic results}\label{section4}

In this section, we prove Proposition \ref{prop:largep}. The proof is essentially independent
of our earlier results, although for convenience we will reference some of our 
earlier formulas. We begin with a well-known lemma (see, e.g., \cite[p. 256]{Cohen1970} for a proof).

\begin{lemma}
\label{lemma:asymptoticfactorization}
Let $f \in \F_p[x]$ be a monic polynomial of degree $n$, and
$C \subset S_n$ a conjugacy class $($i.e., a cycle type$)$ corresponding to the
partition $d_1 + \cdots + d_t = n$. Let $\lambda(C,p)$ be the probability that 
$f$ factors into irreducible polynomials of degrees $d_1, \ldots, d_t$, respectively. Then
$\lambda(C,p) \to |C|/n!$ as $p \to \infty$. 
\end{lemma}

If $\sigma = (d_1^{e_1}\,d_2^{e_2}\,\cdots\,d_t^{e_t})
\in \stypes(n)$ is a splitting type of degree $n$, then by \eqref{numberofpolynomialsoffixedsplittingtype}, we have that 
$N_\sigma$ is a polynomial in $p$ of degree $\sum_{i=1}^t d_i$. Therefore, if 
$e_i>1$ for at least one $i\in\{1,2,\ldots,t\}$, then

\[
\lim_{p\rightarrow\infty}\dfrac{N_{\sigma}}{p^n}=0.
\]
By~\eqref{eqn:alpha-sigma}, to compute $\lim_{p\rightarrow\infty}\alpha(n,d)$, it thus suffices to consider only $\sigma
\in \stypes(n)$ that correspond to factorizations without multiple factors, i.e., to partitions $d_1 + \cdots + d_t = n$ of $n$. It is sufficient to consider only those squarefree polynomials modulo $p$ that have $r\geq d$ distinct  roots (since all of these roots lift by Hensel's lemma), where each such polynomial is  weighted by $\binom{r}{d}$. By Lemma \ref{lemma:asymptoticfactorization}, we wish to count all permutations in $S_n$ with $r$ fixed points, where each such permutation is  weighted by~$\binom{r}{d}$. 
The total weighted number of such permutations is 
$\binom{n}{d}(n-d)!=\frac{n!}{d!}$,  because we can choose $d$ fixed~points in~$\{1,2,\ldots,n\}$,  and then randomly permute the other $n-d$ numbers. It follows that
\begin{equation}
\label{alphalimit}
\lim_{p\rightarrow\infty}\alpha(n,d)=\dfrac{1}{n!}\dfrac{n!}{d!}=\dfrac{1}{d!}.
\end{equation}

By \eqref{getrho-new}, we  
have $\lim_{p\rightarrow \infty}\rho(n,d)=\lim_{p\rightarrow \infty}\rho(n,d,n)$. Either directly from the definitions, or as 
a special case of~\eqref{eqn:rho-cond}, we have $\rho(n,d,n) = \alpha(n,d)$. Therefore, 
\[
\lim_{p\rightarrow \infty}\rho(n,d)=\lim_{p\rightarrow \infty}\alpha(n,d)=\dfrac{1}{d!},
\]
proving
Proposition \ref{prop:largep}(a) for $\rho$ and $\alpha$.

Using \eqref{exp-prob}, and its analogue for $\alpha^*$, we then have
\[
\lim_{p\rightarrow \infty}\rho^*(n,r) = \lim_{p\to\infty}\alpha^*(n,r)=\sum_{d=0}^n (-1)^{d-r} \binom{d}{r} \dfrac{1}{d!}=\frac{1}{r!} \sum_{d=0}^{n-r} (-1)^{d}\dfrac{1}{d!},
\]
proving
Proposition \ref{prop:largep}(b).

To prove the large $p$ limits involving $\beta$, we note that if $d = n-1$ or $d=n$,  
then  \eqref{betasfromalphas} is just $$\beta(n,d)=p^{-\binom{n}{2}}\alpha(n,d),$$ while if  
$d < n-1$, then  Equation~\eqref{betasfromalphas} takes the shape 
\[
\beta(n,d)=p^{-\binom{d+1}{2}}\alpha(d,d)+O(p^{-\binom{d+1}{2}-1}).
\]
From the previous two equations and~\eqref{alphalimit}, we see that
\[
\lim_{p\rightarrow \infty}p^{\binom{n}{2}}\beta(n,n)=\frac{1}{n!}  \quad\mbox{ and }\quad
\lim_{p\rightarrow \infty}p^{\binom{d+1}{2}}\beta(n,d)=\frac{1}{d!}\mbox{ for $d<n$},
\]
proving Proposition \ref{prop:largep}(a) for $\beta$. 

The analogue of \eqref{exp-prob} for $\beta^*$ 
shows that for $r \le n-2$,  we have 
\[ \lim_{p\rightarrow \infty}p^{\binom{r+1}{2}}\beta^*(n,r)=\frac{1}{r!}.\] 
Since $\beta^*(n,n) = \beta(n,n)$, this
completes the proof of Proposition~\ref{prop:largep}(c). 
Note that  $\beta^*(n,n-1)=0$, so there is no need to compute the limits in this case.

If we take  $r=0$ in Proposition~\ref{prop:largep}, we see that 
$$\lim_{p \to \infty} \rho^*(n,0) = \sum_{d=0}^n (-1)^d/d!. $$
The reader may recognise this as the answer to the derangements problem,
i.e., the probability that a random permutation on $n$ letters has no fixed
point. This is the case  because, by Lemma~\ref{lemma:asymptoticfactorization}, monic polynomials without $\Q_p$-roots correspond,  in the large $p$ limit, to permutations without fixed points. 
Similarly, the limit  $\lim_{p \to \infty} \rho^*(n,r) = (1/r!) \sum_{d=0}^{n-r} (-1)^d/d! $ is equal to the probability that a random permutation on $n$ letters has exactly $r$ fixed points. 

\subsection*{Acknowledgments}
We thank the CMI-HIMR Summer School in Computational Number Theory
held at the University of Bristol in June 2019, where this work began.
We also thank Xavier Caruso for kindly sharing with us an earlier
draft of his paper~\cite{Caruso}, and Jordan Ellenberg, Hendrik Lenstra, Steffen
M\"uller, Bjorn Poonen, Lazar Radi\v{c}evi\'c, Arul Shankar, and Jaap
Top for many helpful conversations.

The first author was supported by a Simons Investigator Grant and NSF
grant~DMS-1001828.  The second author was supported by the Heilbronn
Institute for Mathematical Research. The fourth author was supported
in part by DFG-Grant MU 4110/1-1.

We thank the referees for a careful reading of our paper, and for
providing the additional references at the end of
Section~\ref{subsec:previous}.


\begin{thebibliography}{10}

\bibitem{BCFJK}
  M.~Bhargava, J.~E.~Cremona, T.~A.~Fisher, N.~G.~Jones, and  J.~P.~Keating,
  What is the probability that a random integral quadratic form in $n$ variables has an integral zero?,
  {\it Int.~Math.~Res.~Not.}~\textbf{2016}, Issue 12 (2016),
  3828--3848. \url{https://doi.org/10.1093/imrn/rnv251}.

\bibitem{BP}
  A.~Bloch and G.~P\'olya,
  On the roots of certain algebraic equations,
  {\it Proc.~Lond.~Math.~Soc.} {\bf 33} (1932), 102--114.

\bibitem{Buhler-et-al}
  J.~Buhler, D.~Goldstein, D.~Moews, and J.~Rosenberg,
  The probability that a random monic $p$-adic polynomial splits,
  {\it Exper.~Math.} \textbf{15:1} (2006), 21--32.

\bibitem{Caruso}
  X.~Caruso,
  Where are the zeroes of a random $p$-adic polynomial?
  Preprint, October 2021. Available at
  \url{http://xavier.caruso.ovh/papers/publis/randompoly.pdf}.

\bibitem{CFB}
  T.~Church, J.~S.~Ellenberg and B.~Farb,
  Representation stability in cohomology and asymptotics for families
  of varieties over finite fields, {\em Algebraic topology: applications
    and new directions}, {\it Contemp. Math., 620, Amer. Math. Soc.},
  Providence, RI, 2014, 1--54.

\bibitem{Cohen1970}
  S.~D.~Cohen,
  The distribution of polynomials over finite fields,
  {\it Acta Arith.}  \textbf{17} (1970), 255--271.

\bibitem{DPSZ}
  A.~Dembo, B.~Poonen, Q.~Shao, and O.~Zeitouni,
  Random polynomials having few or no real zeros,
  {\it J.~Amer.~Math.~Soc.} {\bf 15} (2002), 857--892.

\bibitem{DL}
  J.~Denef and F.~Loeser,
  Definable sets, motives and $p$-adic integrals,
  {\it J. Amer. Math. Soc.} \textbf{14} (2001), no. 2, 429--469.

\bibitem{DM}
  J.~Denef and D.~Meuser,
  A functional equation of Igusa's local zeta function,
  {\it Amer. J. Math.} \textbf{113} (1991), no. 6, 1135--1152.

\bibitem{Evans}
  S.~Evans, The expected number of zeros of a random system of 
  $p$-adic polynomials, {\it Elec.~Comm.~Prob.} {\bf 11} (2006), 278--290. 

\bibitem{Kac}
  M.~Kac,
  On the average number of real roots of a random algebraic equation,
  {\it Bull.~Math.~Amer.~Soc.}, {\bf 49} (1943), 314--320.

\bibitem{KL}
  A.~Kulkarni and A.~Lerario,
  $p$-adic integral geometry,
  {\it SIAM J. Appl. Algebra Geom.} \textbf{5} (2021), no. 1, 28--59.

\bibitem{Limmer}
  D.~J.~Limmer,
  {\it Measure-equivalence of quadratic forms},
  Ph.D.~Thesis, Oregon State University, 1999.

\bibitem{LO}
  J.~E.~Littlewood and A.~C.~Offord,
  On the number of real roots of a random algebraic equation (i),
  {\it J.~London Math.~Soc.} {\bf 13} (1938), 288--295.

\bibitem{LO2}
  J.~E.~Littlewood and A.~C.~Offord,
  On the number of real roots of a random algebraic equation (ii),
  {\it Proc.~Camb.~Phil.~Soc.} {\bf 35} (1939), 133--148.

\bibitem{LO3}
  J.~E.~Littlewood and A.~C.~Offord,
  On the number of real roots of a random algebraic equation (iii),
  {\it Rec.~Math. [Mat.~Sbornik]} {\bf 54} (1943), 277--286.

\bibitem{Maslova}
  N.~B.~Maslova,
  On the variance of the number of real roots of random polynomials,
  {\it Theory of Probability \& Its Applications} {\bf 19} (1974), 35--52.

\bibitem{Maslova2}
  N.~B.~Maslova,
  On the distribution of the number of real roots of random
  polynomials,
  {\it Theory of Probability \& Its Applications} {\bf 19} (1975), 461--473.

\bibitem{NV}
  O.~Nguyen and V.~Vu,
  Random polynomials: central limit theorems for the real roots,
  {\em Duke Math. J.} {\bf 170} (2021), no. 17, 3745--3813.

\bibitem{P}
  J.~Pas,
  Uniform $p$-adic cell decomposition and local zeta functions,
  {\it J. reine angew. Math.}, \textbf{399} (1989), 137--172.

\bibitem{dSL}
  M.~P.~F.~du Sautoy and A.~Lubotzky,
  Functional equations and uniformity for local zeta functions of nilpotent groups,
  {\it Amer. J. Math.} \textbf{118} (1996), no. 1, 39--90.

\bibitem{Shmueli2021}
R. Shmueli, The expected number of roots over the field of $p$-adic numbers,
\url{arXiv:2101.03561v1} ,  \,Jan.~2021,
to appear in {\it Int.~Math.~Res.~Not.}

\bibitem{Weiss2013}
  B.~L.~Weiss,
  Probabilistic Galois theory over $p$-adic fields,
  {\it J.~Number~Theory} \textbf{133:5} (2013), 1537--1563.
\end{thebibliography}
\end{document}